\documentclass[12pt]{article}
\usepackage[top=1in, bottom=1in, left=1in, right=1in]{geometry}
\pdfoutput=1 

\usepackage{amsfonts}
\usepackage{amsmath}
\usepackage{amssymb}
\usepackage{amsopn}
\usepackage{graphicx}
\usepackage{algorithmic}
\usepackage{placeins}
\usepackage{xcolor,etoolbox}
\usepackage{hyperref}
\usepackage{cleveref}
\usepackage{amsthm}

\newtheorem{theorem}{Theorem}[section]

\numberwithin{equation}{section}

\newcommand{\TheTitle}{Consistency Analysis for Massively Inconsistent Datasets in Bound-to-Bound Data Collaboration}

\title{{\TheTitle}\thanks{This work was supported by the U.S. Department of Energy, National Nuclear Security Administration, 
under Award Number DE-NA0002375.}}

\author{
  Arun Hegde\footnotemark[2]  
  \and
  Wenyu Li\footnotemark[2]
  \and
  James Oreluk\footnotemark[2]
  \and
  Andrew Packard\footnotemark[2]
  \and
 Michael Frenklach\thanks{Department of Mechanical Engineering, University of California, Berkeley, CA 94720-1740
    (arun.hegde@berkeley.edu, wenyuli@berkeley.edu, jim.oreluk@berkeley.edu, apackard@berkeley.edu, frenklach@berkeley.edu).}
}
\date{} 

\begin{document}

\maketitle

\begin{abstract}
Bound-to-Bound Data Collaboration (B2BDC) provides a natural framework for addressing both forward and inverse uncertainty quantification problems. In this approach, QOI (quantity of interest) models are constrained by related experimental observations with interval uncertainty. A collection of such models and observations is termed a dataset and carves out a feasible region in the parameter space. If a dataset has a nonempty feasible set, it is said to be consistent. In real-world applications, it is often the case that collections of models and observations are inconsistent. Revealing the source of this inconsistency, i.e., identifying which models and/or observations are problematic, is essential before a dataset can be used for prediction. To address this issue, we introduce a constraint relaxation-based approach, entitled the vector consistency measure, for investigating datasets with numerous sources of inconsistency. The benefits of this vector consistency measure over a previous method of consistency analysis is demonstrated in two realistic gas combustion examples. 
\end{abstract}



\section{Introduction}
Computational models of complex physical systems must account for uncertainties present in the model parameters, model form, and numerical implementation. Validation of, and prediction from, such models generally requires calibrating unknown parameters based on experimental observations. These observations are uncertain due to the physical limitations of the experimental setup and measuring equipment. In recent years, the topics of verification and validation of complex simulations have undergone much scrutiny  (e.g., \cite{NRCreport, oberkampfbook}), with a particular emphasis on understanding how uncertainty in both models and experimental data are used to inform prediction. Still, validating large-scale models with heterogeneous data, i.e., data of varying fidelity from a multitude of sources, remains a challenge.

The general tenet of the scientific method requires that a proposed model be validated through comparison with experimental data. Ideally, a valid model is one that agrees with the totality of the available data. In practice, this agreement is usually judged by numerical differences between quantities of interest (QOIs) extracted from model predictions and measured data. For instance, Oberkampf and Roy \cite[ch.12]{oberkampfbook} discuss the concept of a \textit{validation metric} as a rigorous means to quantify simulation and experimental differences. Several common strategies for model validation and prediction are probabilistic in nature and employ a Bayesian framework. An example of this can be found in Bayarri et al. \cite{Bayarri071, Bayarri072}, which builds on the seminal work by Kennedy and O'Hagan \cite{OHagan01}. In certain scenarios, however, a less nuanced description of uncertainty can be useful. One such specification, where uncertainty is modeled by set membership constraints, is present in a number of fields, including robust control \cite{packard91,smith92}, robust optimization \cite{Ben09}, engineering design \cite{crespo08}, and computational biology \cite{orth10,price04,palssonbook}. The approach we follow in the present study is Bound-to-Bound Data Collaboration (B2BDC), where uncertainty is modeled deterministically and the notion of validity is encapsulated in the \textit{consistency measure}. 

The B2BDC framework casts the problem of model validation in an optimization setting, where uncertainties are represented by intervals. Within a given (physical) model, parameters are constrained by the combination of prior knowledge and uncertainties in experimental data \cite{fps02,seiler06,russi10}. A collection of such constraints is termed a \textit{dataset} and determines a feasible region in the parameter space. If the feasible region is nonempty, the dataset is  \textit{consistent}---a parameter configuration exists for which models and data are in complete agreement. The consistency measure, introduced by Feeley et al. \cite{feeley04}, characterizes this region by computing the maximal uniform constraint tightening associated with the dataset (positive for consistent datasets, negative otherwise). This optimization-based approach towards model validation and, more generally, uncertainty quantification (UQ) has found application in several settings, including combustion science \cite{fpsf04,feeley04,russi08,frenklach07,CS35} and engineering \cite{utah}, atmospheric chemistry \cite{smith06}, quantum chemistry \cite{prl14}, and system biology \cite{feeley06,feeley08,gprotein}.

Comparison of deterministic and Bayesian statistical approaches to calibration and prediction was performed in a recent study \cite{bbayes}. It was demonstrated, using an example from combustion chemistry, that the two methods were similar in spirit and yielded predictions that overlapped greatly. The principal conclusion was that when applicable, the ``use of both methods protects against possible violations of assumptions in the [Bayesian calibration] approach and conservative specifications and predictions using [B2BDC]" \cite{bbayes}.
Additionally, it was found that ``[s]hortcomings in the reliability and knowledge of the experimental data can be a more significant factor in interpretation of results than differences between the methods of analysis." B2BDC also shares conceptual similarities with the methodology of Bayesian history matching \cite{craig96,craig97,vernon10,vernon14}, the difference being that B2BDC works with bounds while history matching retains a probabilistic framework. Both approaches seek to identify regions of the parameter space where there is model-data agreement, flagging the absence of such a region as an indication of discrepancy. The two approaches use different types of surrogate models, statistical emulators for history matching and polynomial response surfaces for B2BDC.

When accumulating data from diverse sources, it is not uncommon to face discrepancies between observed measurements and corresponding model predictions. Oftentimes, a model is only capable of replicating a subset of the experimental measurements and not the whole. B2BDC identifies such a dataset as being inconsistent, implying no single parameter vector exists for which each model-data constraint is satisfied. This mismatch suggests that there are certain constraints which have been misspecified, either through incorrect model form or flawed experimental data. When analyzing an inconsistent dataset, the sensitivities of the consistency measure to perturbations in the various model-data constraints are used to rank the degree to which individual constraints locally contribute to the inconsistency. Relaxing, or even outright deleting, constraints that dominate this ranking provides a starting point to identify model-data constraints responsible for the inconsistency. Iterating this procedure of assessing consistency and modifying high-sensitivity constraints can lead to a circuitous process --- in \Cref{sec:GRIandDLR} we illustrate an example where it cycles in a rather unproductive fashion and leads to excessive constraint modifications. To address this difficulty, we propose a more refined tool to analyze inconsistent datasets, the \textit{vector consistency measure}, that seeks the minimal number of independent constraint relaxations to reach consistency. Introducing additional variables in the form of relaxation coefficients enables an even richer form of consistency analysis, as will be demonstrated on two example datasets: GRI-Mech 3.0 \cite{grimech} and DLR-SynG \cite{dlr}.

The paper is organized as follows. We first present a brief  overview of the B2BDC methodology in \Cref{sec:B2B} with a particular focus on reasoning with the consistency measure. In \Cref{sec:GRIandDLR}, we review the GRI-Mech 3.0 and DLR-SynG datasets and highlight the successes and failures of the standard sensitivity-based consistency analysis. Our main results are in \Cref{sec:vcm}, where we present the vector consistency measure and motivate the inclusion of relaxation coefficients; a complete vector consistency analysis of the GRI-Mech 3.0 and DLR-SynG datasets is presented in \Cref{sec:caseStudy}. We conclude with a summary in \Cref{sec:summary} and suggest a new protocol for model validation with B2BDC. 

\section{Bound-to-Bound Data Collaboration (B2BDC)}
\label{sec:B2B}
\subsection{Datasets and consistency}
In B2BDC, models, experimental observations, and parameter bounds are taken as ``tentatively entertained", in the spirit of Box and Hunter \cite{Box65}. Consistency quantifies a degree of agreement among the trio, formally within the concept of a dataset. Let $\{M_e(x)\}_{e=1}^N$ be a collection of models defined over a common parameter space, where the $e$\textsuperscript{th} model predicts the $e$\textsuperscript{th} QOI. Further, assume that prior knowledge on the uncertain parameters $x \in \mathbb{R}^n$ is available and encoded by $l_i \leq x_{i} \leq u_i$ for $i = 1,\hdots,n$. Thus, the parameter vector $x$ lies in a hyper-rectangle $\mathcal{H}$. An experimental observation of the $e$\textsuperscript{th} QOI comes in the form of an interval $[L_e, U_e]$, corresponding to uncertainty in observation either from experiment or assessed by a domain expert. To each individual QOI, we can associate a feasible set of parameters on which the corresponding model matches the data
\begin{equation}
\label{eq:mdcFeas}
\mathcal{F}_e := \{x \in \mathcal{H}: L_{e} \leq M_{e} (x) \leq  U_{e} \}.
\end{equation}
A system of such model-data constraints, with prior bounds on the input parameters, constitutes a dataset. Assertions expressing additional knowledge or belief are incorporated within this framework. For instance, one may specify relationships among parameters, e.g., $x_3  \geq 7 x_1$, or relationships among QOIs, e.g., $M_1(x) - M_2(x) \leq  4$, as additional constraints in the dataset.

The set of parameters which collectively satisfy all constraints form the feasible set of the dataset, 
\begin{equation}
\label{eq:mainFdef}
\mathcal{F} := \cap_{e=1}^N \mathcal{F}_e = \{ x \in \mathcal{H} :   L_e \leq M_{e} (x) \leq  U_e \ \text{for} \ e = 1,\hdots, N \}.
\end{equation}
If a parameter vector is not in the feasible set, then it violates the prior information and/or at least one of the model-data constraints. If the feasible set is empty, then no parameter vector can lead to agreement between the models and  the data; the models and the data are fundamentally at odds. Conversely, if the feasible set is nonempty, then any feasible point leads to agreement between models and data.

If the feasible set is nonempty, the dataset is said to be consistent. Gauging feasibility cannot be accomplished simply by comparing the number of constraints to the dimension of the parameter space, as the constraints are usually nonlinear and given by inequalities. The following optimization scheme was introduced in \cite{feeley04} to provide a quantitative measure of a given dataset's consistency. 
\begin{equation}
\label{eq:scm}
\begin{aligned}
C_{\text{D}}:= \max_{\gamma,x \in \mathcal{H}} \quad &\gamma \\
\text{s.t.} \quad &  L_e + \frac{(U_e - L_e)}{2} \gamma \leq M_e(x) \leq U_e - \frac{(U_e - L_e)}{2} \gamma \\
& \text{for }e=1,\dots,N.   
\end{aligned}                                           
\end{equation}
Here, $\pm\frac{U_e -L_e}{2} \gamma$ represents a symmetric (and normalized, via $\gamma$) tightening of each model-data constraint. Thus if $C_\text{D}$ is positive, all observation bounds can be simultaneously tightened without eliminating the feasible set; the dataset is consistent. If $C_\text{D}$ is negative, the observation bounds are too restrictive and must be expanded in order for a feasible set to exist; the dataset is provably inconsistent. The above formulation simply asks the following question: ``what is the largest uncertainty reduction such that there still exists a feasible parameter vector?'', or equivalently, ``by how much must my observational uncertainties change in order to validate or invalidate the dataset?'' If a dataset is consistent, one can proceed with the scientific inquiry. A probabilistic approach that similarly quantifies consistency is the method of Bayesian history matching developed in Craig et al. \cite{craig96,craig97}. There, \textit{implausibility measures} are used to screen the parameter space, eliminating regions where the mismatch between (surrogate) model evaluations and data is deemed unacceptable by a prescribed tolerance. 

In prior work \cite{feeley04}, $C_\text{D}$ is referred to as a ``measure of dataset consistency". In what follows, $C_\text{D}$ is termed the \textit{scalar consistency measure} (SCM) to differentiate it from its \textit{vector} counterpart in the present study. 

\subsection{Sensitivity}
\label{sec:sense}
For convenience, we introduce the following notations.
\begin{alignat}{2}
& f_u^{(i)} (x) = x_i- u_i, \quad f_l^{(i)} (x)=  -x_i + l_i  
&& \quad \text{ for }i=1,\dots,n \\
& f_U^{(e)} (x, \gamma) =  M_e(x) - U_e + \frac{U_e - L_e}{2} \gamma &&\quad \text{ for }e=1,\dots,N \\
& f_L^{(e)} (x, \gamma) = -M_e(x) + L_e + \frac{U_e - L_e}{2} \gamma                                        
\end{alignat}
The SCM can then be rewritten as
\begin{equation}
\label{eq:consreform}
\begin{aligned}
C_{\text{D}} = \max_{\gamma,x} \quad& \gamma \\
\text{s.t.} \quad & f_u^{(i)} (x) \leq 0, 
& f_l^{(i)} (x) \leq 0, 
\quad & \text{for }i=1,\dots,n.  \\
& f_U^{(e)} (x, \gamma) \leq 0, 
& f_L^{(e)} (x,\gamma) \leq 0,
\quad & \text{for }e=1,\dots,N.   
\end{aligned}                                           
\end{equation}
and equivalently formulated as 
\begin{align}
\label{eq:consLagrange}
\begin{split}
C_{\text{D}} = &\max_{\gamma,x} \min_{\lambda \geq 0} \gamma - \sum_{i=1}^n (\lambda_u^{(i)} f_u^{(i)} (x) + \lambda_l^{(i)} f_l^{(i)} (x)) - \sum_{e=1}^N (\lambda_U^{(e)} f_U^{(e)} (x, \gamma) + \lambda_L^{(e)} f_L^{(e)} (x, \gamma)).
\end{split}                                           
\end{align}
Interchanging the order of maximization and minimization produces $\overline{C}_{\text{D}}$, the dual of the SCM, which provides a suitable upper bound 
\begin{align}
\label{eq:dual}
\begin{split}
C_{\text{D}}  \leq  \overline{C}_{\text{D}} := \min_{\lambda \geq 0}\max_{\gamma,x} \gamma - \sum_{i=1}^n (\lambda_u^{(i)} f_u^{(i)} (x) + \lambda_l^{(i)} f_l^{(i)} (x)) - \sum_{e=1}^N (\lambda_U^{(e)} f_U^{(e)} (x, \gamma) + \lambda_L^{(e)} f_L^{(e)} (x, \gamma)).
\end{split}                                           
\end{align}
To study the sensitivity of this measure to changes in the constraint bounds (i.e., changes in the experimental uncertainty and parameter bounds), consider the SCM of a perturbed dataset
\begin{equation}
\label{eq:conspert}
\begin{aligned}
C_{\text{D}}(\rho):= \max_{\gamma,x} \quad& \gamma \\
\text{s.t.} \quad & f_u^{(i)} (x) \leq \rho_u^{(i)}, 
& f_l^{(i)} (x) \leq \rho_l^{(i)},
\quad & \text{for }i=1,\dots,n.  \\ 
& f_U^{(e)} (x, \gamma) \leq \rho_U^{(e)}, 
& f_L^{(e)} (x,\gamma) \leq \rho_L^{(e)},
\quad & \text{for }e=1,\dots,N.   
\end{aligned}                                           
\end{equation}
where $\rho$ collects the perturbations. Note, $\rho > 0$ implements a relaxation to the associated constraints, e.g., $l_i - \rho_l^{(i)} \leq x_i \leq u_i + \rho_u^{(i)}$, whereas $\rho<0$ tightens the associated constraints. In the same manner
\begin{align}
\begin{split}
C_{\text{D}}(\rho)  \leq  \overline{C}_{\text{D}}(\rho) :=  \min_{\lambda \geq 0}\max_{\gamma,x} \gamma &+ \sum_{i=1}^n (\lambda_u^{(i)} (\rho_u^{(i)} - f_u^{(i)} (x)) + \lambda_l^{(i)} (\rho_l^{(i)} - f_l^{(i)} (x))) \\
&+ \sum_{e=1}^N (\lambda_U^{(e)} (\rho_U^{(e)} - f_U^{(e)} (x, \gamma)) + \lambda_L^{(e)} (\rho_L^{(e)} - f_L^{(e)} (x, \gamma))).
\end{split}                                           
\end{align}
Let $\lambda^\star$ be the collection of Lagrange multipliers which are optimal with respect to the unperturbed dual $\overline{C}_{\text{D}}(0)$. Then, $\lambda^\star$ is suboptimal with respect to the perturbed dual. Hence, for all $\rho$
\begin{align}
\label{eq:senseIneq}
\begin{split}
C_{\text{D}}(\rho)  \leq \overline{C}_{\text{D}}(\rho) \leq \overline{C}_{\text{D}}(0) &+ \sum_{i=1}^n \lambda_u^{\star (i)} (u_i - l_i) \frac{\rho_u^{(i)}}{(u_i - l_i)}  + \sum_{i=1}^n \lambda_l^{\star (i)} (u_i - l_i) \frac{\rho_l^{(i)}}{(u_i - l_i)} \\
& + \sum_{e=1}^N \lambda_U^{\star (e)} (U_e -L_e) \frac{\rho_U^{(e)}}{(U_e-L_e)} + \sum_{e=1}^N \lambda_L^{\star (e)} (U_e - L_e)  \frac{\rho_L^{(e)}}{(U_e-L_e)}
\end{split}                                           
\end{align}
where we have normalized the perturbations to reflect relative changes (as opposed to absolute changes) in the corresponding constraint's uncertainty interval, thus making them better suited for comparison. The above inequality shows that the perturbed SCM is bounded above by an expression affine in the perturbation, with nonnegative slopes $\lambda^{\star(i)}(u_i - l_i) \geq 0$ and $\lambda^{\star(e)}(U_e - L_e) \geq 0$. We will refer to these slopes as sensitivities (of the consistency measure to perturbations in the observation uncertainty). Large sensitivities indicate that the SCM is highly responsive to corresponding constraint tightenings and small sensitivities indicate that the SCM is rather unresponsive to corresponding constraint relaxations. Relaxing only constraints with zero-valued sensitivities in an inconsistent dataset where $\overline{C}_D (0) < 0$ can never lead to consistency since $C_D(\rho) \leq \overline{C}_D (0) < 0$ regardless of the perturbation.

Computing the dual $\overline{C}_{\text{D}}$ and its associated Lagrange multipliers can often be accomplished using the tools of convex optimization \cite{boydbook}. In the present study,  models $M_e(x)$ take the form of general quadratics and the dual computation reduces to solving a semidefinite program. In the examples presented below, each quadratic model acts as a surrogate for it's underlying computer model. Details on utilizing B2BDC with quadratics and more general classes of models, such as rational quadratics, can be found in \cite{feeley08,russi08}. In particular, \cite{feeley08} presents specialized derivations for sensitivity analysis relevant to the quadratic datasets described in this paper. 

\subsection{Using the scalar consistency measure and sensitivities}
\label{sec:usingSCM}
If a dataset is provably inconsistent, with $C_{\text{D}} < 0$ ($\overline{C}_{\text{D}} <0$ is a sufficient condition), the sensitivities of the SCM provide a means of flagging model-data constraints that may be responsible for the inconsistency. Intuitively, scientific conclusions should be robust relative to small perturbations of observed data and uncertainty. Thus, loosely speaking, we should not expect the SCM $C_{\text{D}}(\rho)$ to change much for small $\rho$. The computational aspects of such an analysis are fully detailed in prior work \cite{feeley04, feeley08}.

The strategy employing SCM is to identify the constraint (or constraints) with the highest sensitivity, modify that constraint or remove it completely, and recompute the SCM for the new dataset. This process is iterated until consistency is reached. However, as will be demonstrated in the subsequent section, such an approach becomes rather ineffective when dealing with what we will refer to as \textit{massively inconsistent} datasets --- datasets with numerous contributors to inconsistency. 

\section{Motivating examples: GRI-Mech 3.0 and DLR-SynG}
\label{sec:GRIandDLR}

\subsection{GRI-Mech 3.0 and DLR-SynG datasets}
In the following paragraphs, we provide a brief overview of two example datasets coming from the field of combustion chemistry. The first dataset, GRI-Mech 3.0, illustrates a successful application of the iterative-SCM methodology. The second dataset, DLR-SynG, demonstrates a major pitfall with this approach, showing that the iteration can result in excessive constraint modifications. The successes and failures of the iterative-SCM method for resolving inconsistent datasets, as displayed by these two examples, motivates the development of an additional tool specific to analyzing inconsistency.        

The first example, the GRI-Mech 3.0 dataset, was originally built to calibrate an underlying kinetic reaction model for pollutant formation in natural gas combustion \cite{grimech}. The kinetic model takes the form of an ODE system representing 325 reversible reactions among 53 different chemical species. The resulting dataset consists of 77 model-data constraints (corresponding to 77 expert-chosen QOIs) in 102 active model parameters. The application of B2BDC to this dataset has been extensively investigated in several prior studies \cite{fps02,feeley04,xy10,bbayes}. 

The second example, the DLR-SynG dataset, is part of an ongoing work to construct a predictive kinetic model for syngas combustion \cite{dlr}. As with GRI-Mech 3.0, DLR-SynG is formulated as an ODE system representing 73 reactions in 17 different chemical species. The resulting dataset is comprised of 159 model-data constraints in 55 uncertain parameters. The experimental uncertainty came in the form of expert-assessed bounds \cite{dlr}. 

\subsection{Consistency analysis of GRI-Mech 3.0 dataset}
An initial application of \Cref{eq:scm} reveals that the SCM lies in the interval $[-0.37, -0.31]$, where the lower bound is a local solution determined using the solver {\fontfamily{pcr}\selectfont fmincon} from MATLAB's Optimization Toolbox \cite{MATLAB} and the upper bound is determined via semidefinite programming. The negative sign of the SCM upper bound certifies that the GRI-Mech 3.0 dataset is inconsistent;  the resulting sensitivities are displayed in \Cref{fig:GRISensitivity}.
\begin{figure}[htbp]
  \centering
  \includegraphics[width =\linewidth]{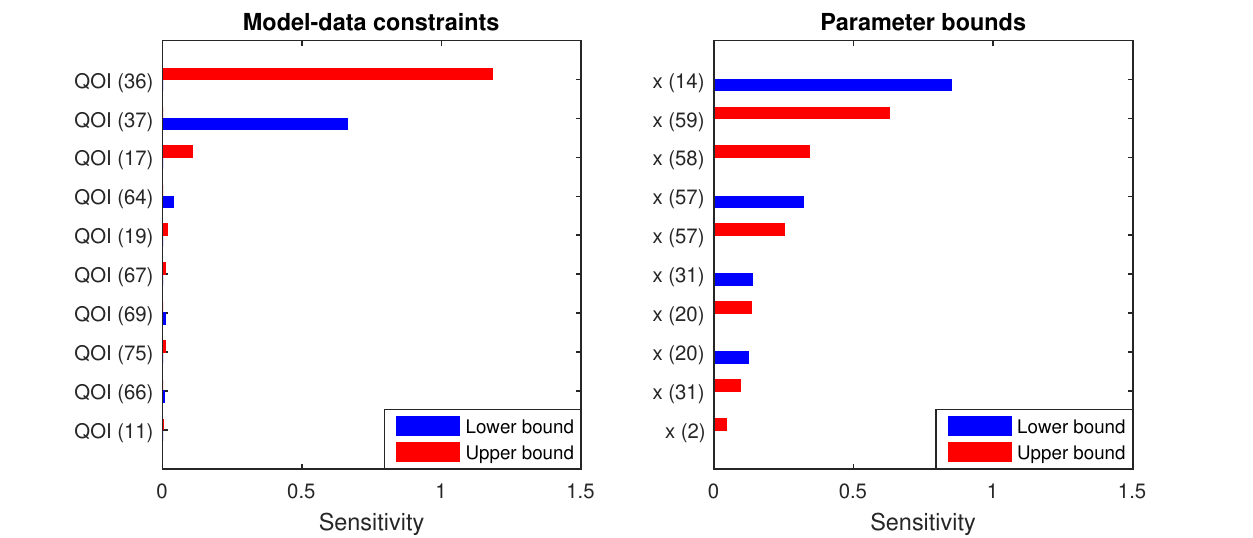}
  \caption{Top ranked sensitivities for the GRI-Mech 3.0 dataset, as described in \Cref{eq:senseIneq}. Left, sensitivities of the SCM to perturbations of the model-data constraints. Right, sensitivities of the SCM to perturbations of the parameter bounds.}
  \label{fig:GRISensitivity}
\end{figure}
It is immediately apparent that relatively few of the model-data constraints have large sensitivities. In particular, the upper bound to the $36$\textsuperscript{th} model-data constraint and the lower bound to the $37$\textsuperscript{th} QOI were flagged by the sensitivity analysis. Removing QOIs $36$ and $37$ resulted in a consistent dataset, with SCM in the interval $[0.14, 0.22]$. A trial-and-error analysis uncovered that the GRI-Mech 3.0 dataset can also be made consistent by removing only constraint $37$. As for parameters, the prior bounds were assumed accurate and we did not admit any modification. 

\subsection{Consistency analysis of DLR-SynG dataset}
The iterative-SCM approach provided a simple and efficient strategy for resolving the GRI-Mech 3.0 dataset. Repeating this process with the DLR-SynG dataset, however, lead to a wholly different experience. An initial assessment of the SCM resulted in the interval $[-2.02,-1.64]$. The corresponding sensitivities are shown in \Cref{fig:DLRSensitivity1}. 
\begin{figure}[htbp]
  \centering
  \includegraphics[width =\linewidth]{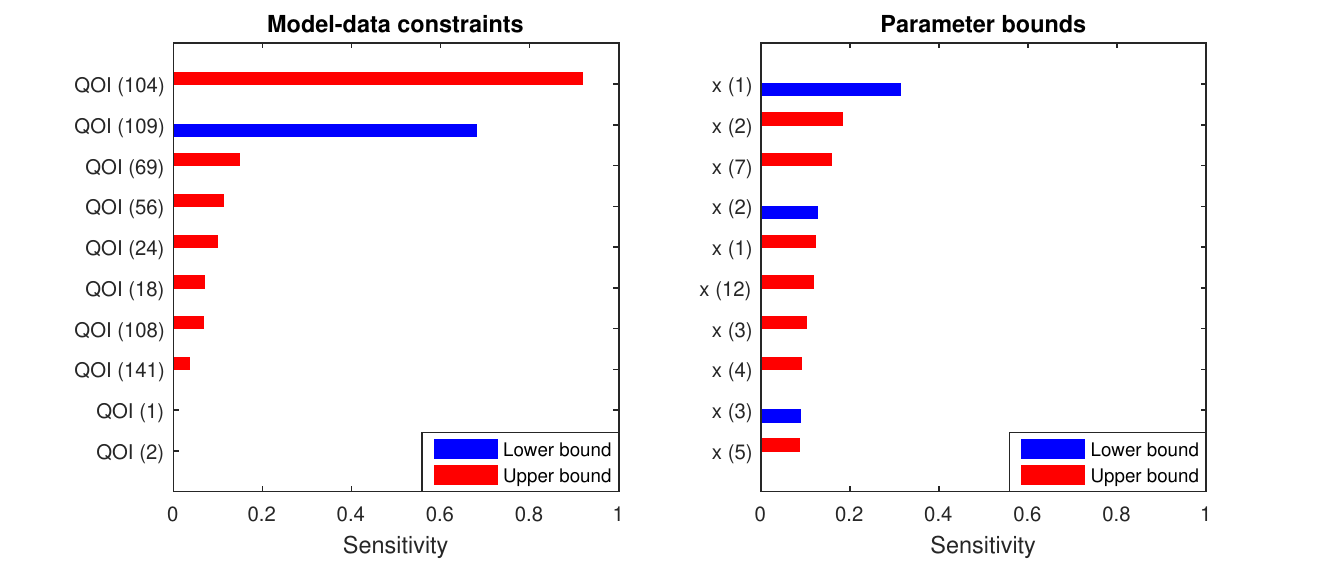}
  \caption{Top ranked sensitivities for the DLR-SynG dataset. Left, sensitivities of the SCM to perturbations of the model-data constraints. Right, sensitivities of the SCM to perturbations of the parameter bounds.}
  \label{fig:DLRSensitivity1}
\end{figure}
At this initial stage, the results look quite similar to the previous case --- relatively few model-data constraints dominate the sensitivity ranking. Deleting the model-data constraint with the largest sensitivity, i.e., QOI $104$, resulted in an inconsistent dataset with SCM  $[-2.02,-1.61]$. The sensitivities of the second consistency assessment are plotted in \Cref{fig:DLRSensitivity2}.
\begin{figure}[htbp]
  \centering
  \includegraphics[width =\linewidth]{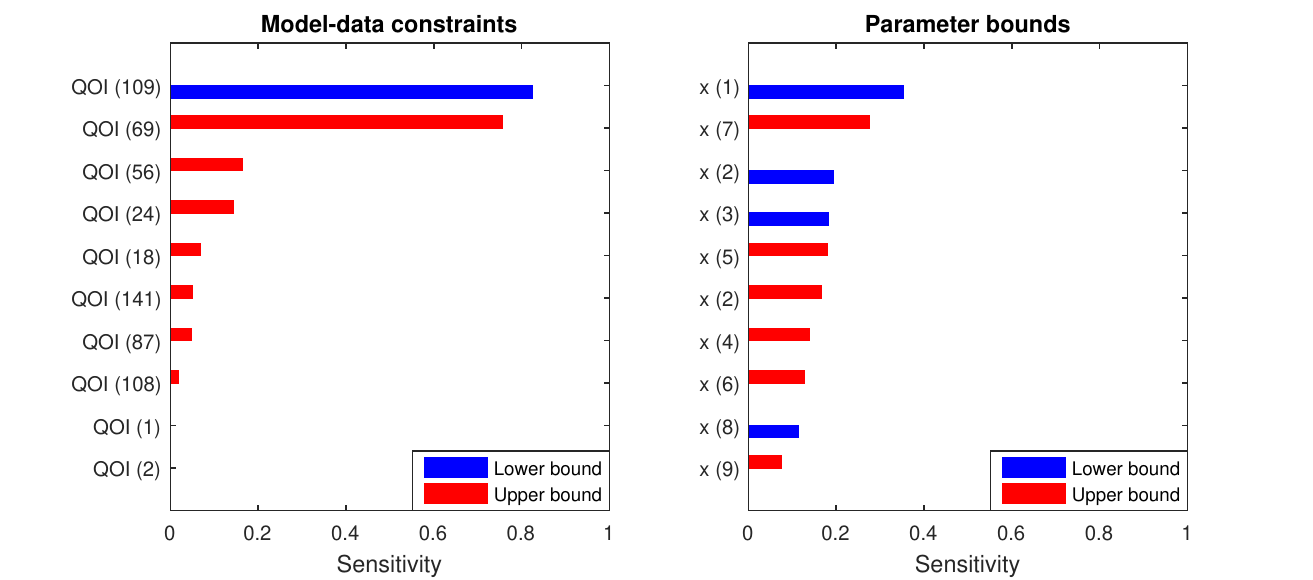}
  \caption{Top ranked sensitivities for the DLR-SynG dataset after the first constraint deletion. Left, sensitivities of the SCM to perturbations of the model-data constraints. Right, sensitivities of the SCM to perturbations of the parameter bounds.}
  \label{fig:DLRSensitivity2}
\end{figure}
Note, the majority of the flagged QOIs were also identified in the first stage of the analysis. QOI 87, however, is new and was previously associated with a near-zero sensitivity value. Similarly, the parameter sensitivities are also slightly different between the first and second stages of the analysis. This process of calculating sensitivities and deleting the  model-data constraint with the highest sensitivity was repeated until the dataset became provably consistent, i.e., the local solution of the SCM was nonnegative, and resulted in the deletion of $73$ model-data constraints. After completing this entire procedure, one might ask the question: ``what would have been different if I had removed the second most sensitive constraint (as opposed to the top most sensitive constraint) in the $k$\textsuperscript{th} stage of the iteration?'' Following this line of thought and deleting the second most sensitive constraint at each iterate led to a total of $56$ constraint removals. At each stage, removing a different constraint may have resulted in different and potentially fewer future constraint removals. An alternative and extreme strategy would be to delete all of the model-data constraints with nonzero sensitivities at each iterate. Doing so leads to an iterative scheme which repeats five times and results in the deletion of $83$ model-data constraints, which does not improve upon the previous results. With each analysis, the concern still remains: ``could consistency be reached with fewer constraint removals?'' After a few stages of the analysis, it becomes apparent that the scale of the inconsistency in DLR-SynG dwarfs that of GRI-Mech 3.0. 
   
\section{Vector consistency}
\label{sec:vcm}
\subsection{Vector consistency measure}
\label{sec:vcm1}
Motivated by the results discussed in the previous section, we explored alternative ways to resolve massively inconsistent datasets. The SCM presented earlier gauges the consistency of a dataset by either uniformly relaxing or tightening all model-data constraints. Now, instead of a single relaxation to all constraints, we consider independent relaxations to each model-data constraint with the aim of finding the fewest number of relaxations required to render the dataset consistent. This leads to the optimization
\begin{equation}
\label{eq:vcmL0}
\begin{aligned}
V_{\|\cdot\|_0} := \min_{x,\Delta_L,\Delta_U, \delta_l, \delta_u} & \|\Delta_L\|_0 + \|\Delta_U\|_0 + \|\delta_l\|_0 + \|\delta_u\|_0 \\
\text{s.t.} 
\qquad & L_e - \Delta_{L}^{(e)} \leq M_e(x) \leq U_e + \Delta_{U}^{(e)} \quad & \text{for } e = 1,...,N \\
& l_i - \delta_{l}^{(i)} \leq x_i \leq u_i + \delta_{u}^{(i)} 
\quad & \text{for } i = 1,...,n 
\end{aligned}
\end{equation}
where the function $\|\cdot\|_0$, sometimes called the 0-norm, expresses the number of nonzero entries of its argument and relaxations ($\Delta, \delta$) have been introduced to both model-data constraints and parameter bounds. If a dataset is consistent, no relaxations are needed and the optimal value is zero. For an inconsistent dataset, solution of the above problem gives the smallest number of bound changes to reach consistency as well as a parameter vector that becomes consistent. Going further and removing the model-data constraints flagged with nonzero relaxations results in a pruned dataset that contains the maximal subcollection of the original model-data constraints that are together consistent. Relaxations in \cref{eq:vcmL0} can be taken as either independent or supplemented with additional constraints accounting for dependencies. For example, one may wish to specify explicit correlations, like $\Delta_U^{(1)} = \Delta_U^{(2)}$.

The above optimization problem is difficult to solve as the 0-norm is non-convex and even linear problems involving this objective are NP-Hard \cite{Natarajan95}. To address the non-convexity, we use the classical technique of replacing the 0-norm with the 1-norm, a well-known convex heuristic for sparsity \cite{donoho04, candes05, brecht10}, producing the optimization
\begin{equation}
\label{eq:vcm}
\begin{aligned}
V_{\|\cdot\|_1} := \min_{x,\Delta_L,\Delta_U, \delta_l, \delta_u} & \|\Delta_L\|_1 + \|\Delta_U\|_1 + \|\delta_l\|_1 + \|\delta_u\|_1 \\
\text{s.t.} 
\qquad & L_e - \Delta_{L}^{(e)} \leq M_e(x) \leq U_e + \Delta_{U}^{(e)} \quad &\text{for } e = 1,...,N\\
& l_i - \delta_{l}^{(i)} \leq x_i \leq u_i + \delta_{u}^{(i)} \quad  &\text{for } i = 1,...,n\\
\end{aligned}
\end{equation}
which we term the \textit{vector consistency measure} (VCM) as we now allow a vector of relaxations. We emphasize that the answer to the above problem should be taken as a theoretical reference point; whether such a relaxation to the experimental bounds is justifiable or whether that relaxation signals a deficiency of the underlying physical model is to be addressed by domain science. What the analysis reveals is that directly implementing the relaxations (expanding the relevant bounds by the determined amount), or removing the model-data constraints with nonzero relaxations, or some combination are all strategies that result in a consistent dataset.  

As a consequence of the above formulation, we note that the VCM will never simultaneously relax both bounds of a single constraint. To illustrate this, without loss of generality, consider datasets with only model-data constraints. Begin with an inconsistent dataset and let $\Delta_L^\star$ and $\Delta_U^\star$ be the VCM-optimal relaxations. Thus, the set
\begin{equation}
\label{eq:newF}
\mathcal{F}' = \{x :  L_e - \Delta_{L}^{\star(e)} \leq M_e(x) \leq U_e + \Delta_{U}^{\star(e)}, \ e = 1,...,N \} 
\end{equation}
is nonempty. Suppose there is a constraint where both relaxations are nonzero. Then there exists a $z \in \mathcal{F}'$ and an index $k$ such that $\Delta_{L}^{\star(k)}$ and $\Delta_{U}^{\star(k)}$ are positive. There are three cases to consider: $M_k(z) = L_k - \Delta_L^{\star(k)}$, $M_k(z) = U_k + \Delta_U^{\star(k)}$, or $L_k - \Delta_L^{\star(k)} < M_k(z) < U_e + \Delta_U^{\star(k)}$. In each case, we find viable relaxations with smaller 1-norms by not including $\Delta_U^{\star(k)}$ (corresponds to case 1), not including $\Delta_L^{\star(k)}$ (corresponds to case 2), or decreasing both relaxations by some small amount (corresponds to case 3). So, relaxations are only applied to a single bound. Hence, the VCM can be rewritten with fewer decision variables,
\begin{equation}
\label{eq:vcmRewrite}
\begin{aligned}
\min_{x,\Delta, \delta} \qquad & \|\Delta\|_1 + \| \delta\|_1 \\
\text{s.t.} 
\qquad & L_e  \leq M_e(x) - \Delta_e \leq U_e \quad &\text{for } e = 1,...,N\\
& l_i \leq x_i - \delta_i \leq u_i  \quad  &\text{for } i = 1,...,n.\\
\end{aligned}
\end{equation}
The sign of the relaxation determines to which bound it is applied. Expressed in this manner, the VCM can also be interpreted as providing a minimal shift in model predictionsand parameters to reach consistency. This is similar to the introduction of a bias term in the probabilistic approach \cite{Bayarri072}.

Moreover, the relaxed constraints will always be met with equality. Again, this is demonstrated by contradiction. Continuing with the notation above, suppose that the constraints with upper bound relaxations are not always met with equality. Then there exists a $z \in \mathcal{F}'$ and an index $k$ such that either $L_k \leq M_k(z) < U_k$ or $U_k \leq M_k(z) < U_k + \Delta_U^{\star(k)}$. So, either the relaxation is unnecessary or it can be made smaller. The same follows for constraints with lower bound relaxations. Hence, the relaxed constraints are met with equality for all parameters in $\mathcal{F}'$. In the context of a VCM analysis, $\mathcal{F}'$ is the feasible set that arises from actually implementing all of the determined relaxations. Thus,  model predictions of the relaxed QOIs are exact on the new feasible set. In practice, $\mathcal{F}'$ may consist of a single point, suggesting no parametric uncertainty and hence no uncertainty in model prediction. Such a result must be regarded with caution, as the newfound certainty comes from a purely mathematical, rather than physical, source. 

Generally, inequality relaxation is a key component of constrained optimization and has been used in different contexts, e.g., the sum of infeasibilities method \cite[p. 580]{boydbook} and the elastic filter \cite[p. 101]{chinneckbook}. Its application to UQ was suggested as early as \cite[p. 41]{feeley08}. The optimal value of the above problem is still difficult to compute as the non-convexity of the models has yet to be addressed. For $(M_e)_{e=1}^N$ described by general quadratic functions, the VCM can be efficiently bounded from below by solving a semidefinite program (\Cref{eq:rvcm} in \Cref{sec:AppendixA}). An upper bound (any local solution) can always be found via generic nonlinear programming techniques applied to \Cref{eq:vcm}, e.g., via MATLAB's {\fontfamily{pcr}\selectfont fmincon}. Note that a positive lower bound on the objective is sufficient to guarantee the inconsistency of a dataset. Any local solution will produce an upper bound on the objective, a feasible set of relaxations, and a parameter vector that becomes feasible once the relaxations are implemented. In instances where the lower bound is non-informative, i.e., equal to zero, the SCM can be used as a supplementary tool to assess the consistency. 

\subsection{Linear examples and counterexamples}
\label{sec:linearExamples}
Although the above formulation for the VCM is attractive, it fails to guarantee some key properties one might wish to have, even in the most basic of cases. In this section, we investigate the VCM problem in the context of linear models and demonstrate manifestations of these issues. The findings motivate the inclusion of relaxation coefficients into the vector consistency framework to allow for additional flexibility.  

Consider the following specialization of the vector consistency problem to linear models, namely, 
\begin{equation}
\label{eq:linVCM}
\begin{aligned}
\min_{x,\delta}& \ \|\delta\|_1 \\
\text{s.t.} & \ Ax \leq b +\delta
\end{aligned}
\end{equation}
where $A \in \mathbb{R}^{m \times n}$ and $b \in \mathbb{R}^m$. Thus, if the constraint $Ax \leq b$ is feasible, the optimal relaxation $\delta^\star$ will be the zero vector. Suppose we start with a feasible dataset $Ax \leq b$, corresponding to some underlying truth, and tighten the constraint by some error vector $t$, made up of 0's and 1's, such that it becomes inconsistent, i.e., $\{x: Ax \leq b_\alpha \} = \varnothing$ for some $\alpha>0$ where $b_\alpha = b-\alpha t$. Note that the introduced error can be either interpreted as an error in the bounds or as a bias in the model. The question posed is as follows: for an increasing $\alpha$, can we guarantee that \Cref{eq:linVCM} will eventually recover the error, i.e., $\delta^\star = \alpha t$? Can we expect to recover the underlying ``true'' feasible set? The answer turns out to be no. This is demonstrated by the following simple counterexample.     

Let $A = \begin{bmatrix} 1.5 \\ -1 \end{bmatrix}$, $b = \begin{bmatrix} 1 \\ 1\end{bmatrix}$, and $t = \begin{bmatrix} 1 \\ 0 \end{bmatrix}$ (the error is only introduced to the first constraint). With these settings, \Cref{eq:linVCM} can be rewritten as,
\begin{equation}
\label{eq:counterEx}
\begin{aligned}
\min_{x,\delta}& \ \|\delta\|_1 \\
\text{s.t.} & \ \begin{bmatrix} 1.5 \\ -1 \end{bmatrix} x - \begin{bmatrix} 1 \\ 1\end{bmatrix} + \alpha \begin{bmatrix} 1 \\ 0 \end{bmatrix}
\leq \begin{bmatrix} \delta_1 \\ \delta_2 \end{bmatrix}
\end{aligned}
\end{equation}
where $\alpha$ is an increasing parameter. For any $\alpha > 2.5$, simple calculation reveals the dataset is inconsistent. The optimal relaxation $\delta^\star$ can be characterized by the intersection of the following two regions,
\begin{equation}
\label{eq:counterExRegion}
\begin{gathered}
 K_\alpha = \{ y \in \mathbb{R}^2: \ \exists \delta \text{ such that } \|\delta\|_1 \leq c^\star, \ \ y \leq \delta \} \\
 G_\alpha = \{y \in \mathbb{R}^2 : y = Ax-b+ \alpha t, x \in \mathbb{R} \},
\end{gathered}
\end{equation}
where $c^\star$ is the smallest value such that the intersection is non-empty (i.e., the solution of \Cref{eq:counterEx}). This intersection is displayed in \Cref{fig:counterEx} for two values of $\alpha$. Although the error has been introduced along the $y_1$ axis, $G_\alpha$ will always intersect $K_\alpha$ along the $y_2$ axis due to its slope. Thus, the VCM will always suggest a relaxation to the second constraint despite the error being introduced to only the first constraint. Simply put, $\delta^\star \neq \alpha t$ for any $\alpha > 0$. Regardless of the magnitude of the error, the VCM will always pick the wrong constraint.
\begin{figure}[htbp]
  \centering
  \includegraphics[width=\linewidth]{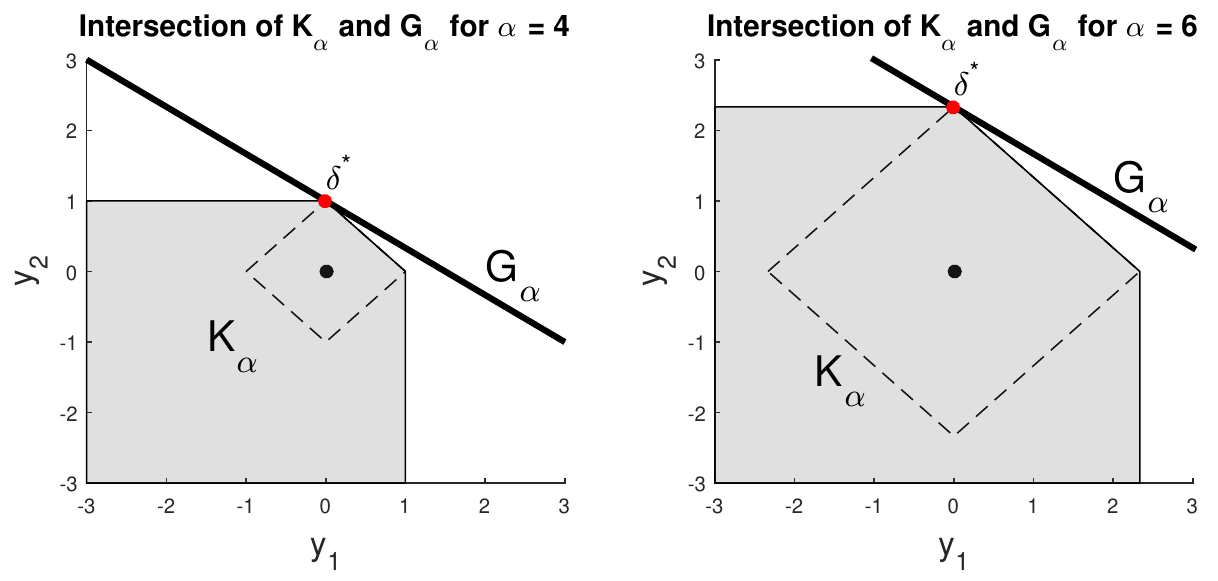}
  \caption{Illustration of the counter example. The dashed lines indicate the 1-norm ball with radius $c^\star$.}
  \label{fig:counterEx}
\end{figure}
\FloatBarrier
The above counterexample illustrates a rather contrived case where the VCM is incapable of recovering a specified error. To explore and test the significance of the result, we investigated the performance of the measure on random instances of \Cref{eq:linVCM}, which can be implemented via linear programming. Let $A \in \mathbb{R}^{50 \times 15}$ have entries selected from a uniform distribution on $[-1,1]$, and let $b \in \mathbb{R}^{50}$ be such that $Ax \leq b$ is feasible. Additionally, let $t \in \mathbb{R}^{50}$ be defined as above with a total of $n_E$ 1's placed in random entries. Let $\alpha>0$ be such that $\{x: Ax \leq b_\alpha \}$ is empty and define the following ratios,    
\begin{equation}
\label{eq:ratios}
\begin{aligned}
& \phi_{E} = \frac{ |\text{supp}(\delta^\star) \cap \text{supp}(t) |} {n_E} \\
& \phi_{\delta} = \frac{ |\text{supp}(\delta^\star) \cap \text{supp}(t) |} {|\text{supp}(\delta^\star)|}
\end{aligned}
\end{equation}
where $\text{supp}(a) = \{i: a_i \neq 0 \}$, $|S|$ returns the cardinality of a set $S$, and $\delta^\star$ is the optimal value of \Cref{eq:linVCM}. The ratio $\phi_E$ measures the fraction of errors correctly identified by the optimization, whereas $\phi_\delta$ measures the number of correctly identified errors relative to the total number of identifications. Both ratios range between zero and one, with $(\phi_E,\phi_\delta) = (1,1)$ indicating perfect identification. Together, these ratios account for both under-identifying and over-identifying the constraints involved in the inconsistency. The results of $10,000$ random trials for $n_E = 1$ and two values of $\alpha$ are displayed in \Cref{fig:hist1}.  
\begin{figure}[htbp]
  \centering
  \includegraphics[width=\linewidth]{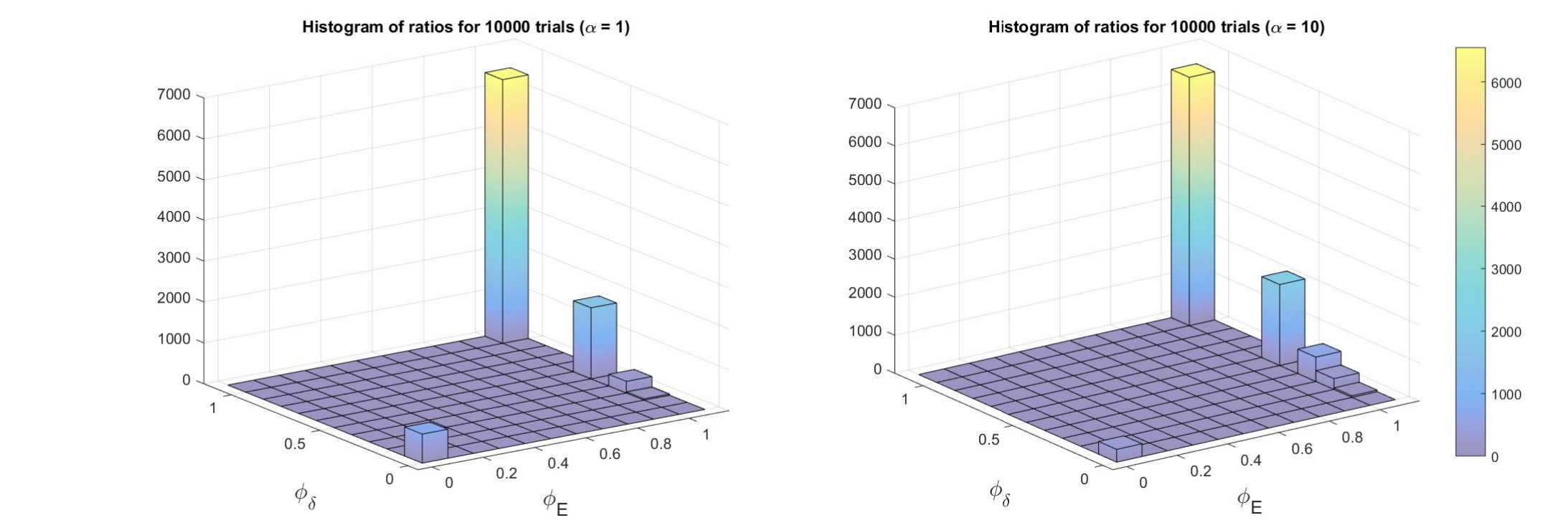}
  \caption{Histogram of ratios for 10,000 random trials with $n_E = 1$.}
  \label{fig:hist1}
\end{figure}
\FloatBarrier
With this setup, the methodology displays perfect identification for the vast majority of trials. With $n_E$ increased to four, however, the identification becomes more challenging as the number of possible explanations for the inconsistency increases.  
\begin{figure}[htbp]
  \centering
  \includegraphics[width=\linewidth]{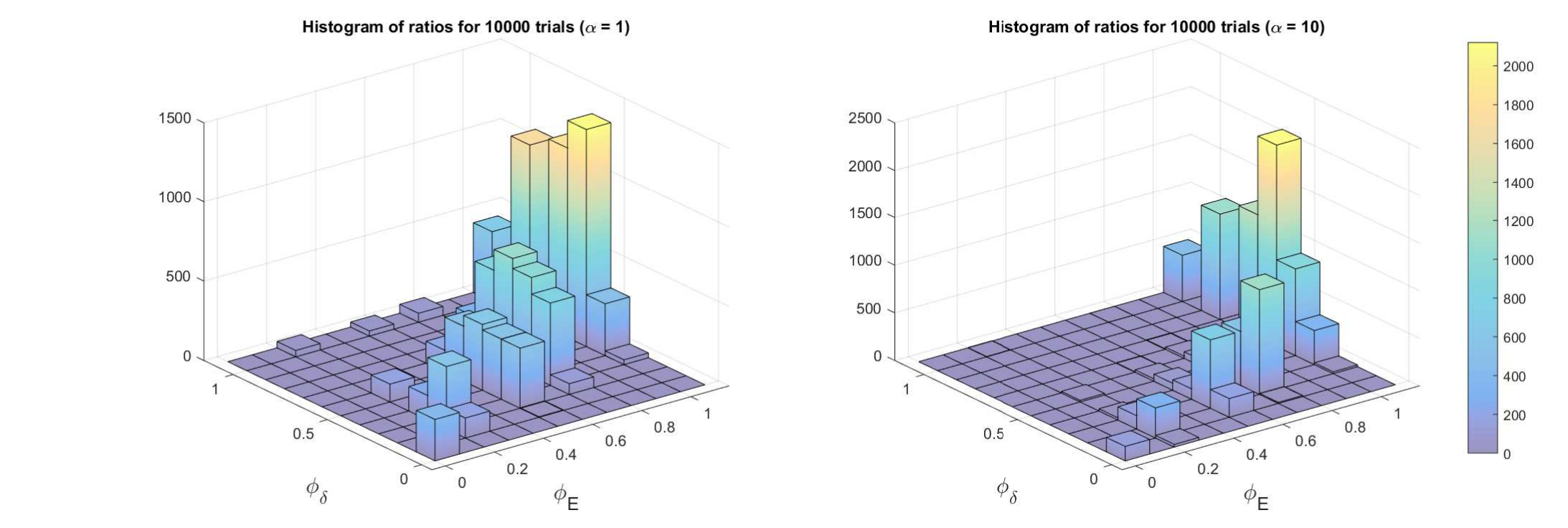}
  \caption{Histogram of ratios for 10,000 random trials with $n_E = 4$.}
  \label{fig:hist2}
\end{figure}
\FloatBarrier
The majority of trials yield $\phi_\delta \approx 0.5$ and $\phi_E = 1$. Thus, the VCM suggests relaxations to all four of the actual errored constraints plus approximately four extra constraints. In a smaller but significant number of trials, the optimization is unable to identify all of the errors and suggests relaxations to other constraints. There are many paths to consistency and in these cases \Cref{eq:linVCM} finds alternative relaxations (other than to the errors in $t$) that have minimal 1-norms. 

In the context of validation, the notion of a single ``correct'' relaxation is somewhat misrepresentative and oversimplifies the issue. As discussed above, for any given inconsistent dataset, there may be a multitude of relaxations that could lead to consistency, with each relaxation leading to different feasible parameter configurations. The presumption of sparsity, which guides the VCM, does not alone provide enough information to pick out this so-called correct relaxation in a reliable manner. To achieve this we would require information from domain experts. This feature can be incorporated into the vector consistency framework through the inclusion of user-specified coefficients to the relaxations. 

\subsection{Augmented vector consistency measure}
When analyzing large-scale simulations, it is often the case that a domain expert has prior knowledge or an opinion that certain experiments may be more reliable than others. Similarly, this notion of reliability is often also present in models and simulations; certain implementations of theory, perhaps representing specific pieces of physics, may be of higher fidelity than other components in a model. Thus, in a B2BDC consistency analysis, model-data constraints need not be treated with equal importance. If a dataset is inconsistent, one should be less willing to relax or modify constraints with high fidelity models and reliable experimental data. In addition to reliability, certain experiments may be more relevant to the intended use of the dataset. These experimental results should also be prioritized in some fashion. This notion of priority can be included in a vector consistency analysis through relaxation coefficients, as displayed in \Cref{eq:avcm}, 
\begin{equation}
\label{eq:avcm}
\begin{aligned}
V_{\|\cdot\|_1}(r) = \min_{x,\Delta_L,\Delta_U, \delta_l, \delta_u} & \|\Delta_L\|_1 + \|\Delta_U\|_1 + \|\delta_l\|_1 + \|\delta_u\|_1 \\
\text{s.t.} \qquad & L_e - R_L^{(e)} \Delta_{L}^{(e)} \leq M_e(x) \leq U_e + R_U^{(e)} \Delta_{U}^{(e)} \quad &\text{for } e = 1,...,N \\
& l_i - r_l^{(i)} \delta_{l}^{(i)} \leq x_i \leq u_i + r_u^{(i)} \delta_{u}^{(i)} \quad &\text{for } i = 1,...,n
\end{aligned}
\end{equation}
where $r$ collects the coefficients $\{R_L,R_U,r_l,r_u\}$. Note, we have returned to the formulation in \Cref{eq:vcm} to highlight that different bounds in a given QOI may be assigned different coefficients. In the above framework, the optimization will prefer relaxing constraints with large relaxation coefficients as that action has greater impact on the objective. In contrast, constraints with small coefficients are protected from relaxation. For instance, setting the parameter coefficients $r_l^{(i)}$ and $r_u^{(i)}$ to zero implies the parameter bounds are absolute and cannot be expanded to reach consistency. Setting model-data relaxation coefficients $R_L^{(e)}$ and $R_U^{(e)}$ to zero would imply that the corresponding constraint bounds are immutable. In essence, these coefficients specify the degree of flexibility afforded to certain constraints. The influence of differing relaxation coefficients can be visualized by considering the previous counterexample in \Cref{fig:counterEx}. The coefficients will distort the 1-norm ball by stretching along directions with large values. Hence, specifying a larger coefficient for relaxations to the first constraint will allow the error to be recovered. See \Cref{sec:AppendixC} for details.
 
In general, relaxation coefficients should characterize the relative importance of certain model-data constraints over others. In the absence of prior knowledge on the constraints, there are several standard coefficient schemes one can turn to that may be informative. For instance, the original vector consistency scheme in \Cref{eq:vcm} is recovered by setting all coefficients to one. Several common coefficient configurations and their interpretations are summarized in \Cref{tab:relaxCoef} for a generic constraint.
\begin{table}[htbp]       
\caption{Example relaxation coefficients for a generic constraint:  $L - r_L \delta_L \leq f(x) \leq U+r_U \delta_U$}
\begin{center}
\begin{tabular}{|c|p{9.1cm}|} \hline
Coefficient & $\delta_L, \delta_U$ --- Interpretation \\ \hline
$r_L = r_U =1$  & Absolute change in bound (``\textit{unit} coefficient'') \\ \hline
$r_L = r_U =(U-L)$ & $\delta_L+\delta_U \sim$ percent expansion in uncertainty interval (``\textit{interval} coefficient'')  \\ \hline
$r_L =|L|, r_U = |U|$ & Percent decrease/increase in lower/upper bound (``\textit{bound} coefficient'')\\ \hline
$r_L=r_U=0$ & No relaxation permitted (``\textit{null} coefficient'') \\ \hline
\end{tabular}
\end{center}
\label{tab:relaxCoef}
\end{table}
\FloatBarrier

\subsection{Using the vector consistency measure}
\label{sec:usingVCM}
If a dataset is consistent, the VCM is less informative than the SCM. On the other hand, if a dataset is inconsistent, the VCM returns a list of relaxations that, if implemented, result in a consistent dataset. Based on the number of nonzero relaxations and their associated magnitudes, one can get a sense of how far a dataset is from consistency. Does a single constraint need to be expanded by some large amount? Or, do many constraints need to be changed by small amounts? Phenomena like massive inconsistency, where many constraints require non-negligible expansions, are immediately recognizable by studying the spectrum of relaxations.

After performing the VCM analysis, multiple actions can be taken. Expanding the experimental bounds by the determined relaxations, removing the flagged model-data constraints, or some combination all lead to consistent datasets. Alternatively, refining the underlying  model, guided by the flagged QOIs and corresponding relaxations, may also lead to consistency. As discussed in \Cref{sec:vcm1}, the decision of how to proceed is to be resolved by domain scientists. In Feeley et al. \cite{feeley04}, model inadequacy was primarily attributed to the parameters and certain experimental bounds were revised. In Slavinskaya et al. \cite{dlr}, the (temporary) removal of QOIs was warranted due to perceived deficiencies in their instrument models rather than the underlying reaction model. In Iavarone et al. \cite{Iavarone17}, a new model form consistent with experimental data was proposed to replace a previously inconsistent model.

The inclusion of relaxation coefficients into the VCM aids in this decision-making process. The VCM analysis can be repeated in a trial-and-error fashion, using different relaxation coefficients, to explore a dataset's inconsistency. If the VCM detects the presence of a massive inconsistency, then this suggests possibly severe limitations in the underlying model as it is unlikely that a large number of diverse experiments would be in error. Model development is an iterative  process  that involves comparison of model prediction and data \cite{boxbook}. The VCM (and SCM) provides a formalization of ``comparing model to data" in this process.

In the following section, this new tool for consistency analysis is demonstrated on the GRI-Mech 3.0 and DLR-SynG datasets for various coefficient configurations. The results can be directly compared to the sensitivity-driven iterative-SCM discussed in \Cref{sec:GRIandDLR}.  

\section{Vector consistency analysis of example datasets}
\label{sec:caseStudy}
\subsection{GRI-Mech 3.0}
An initial application of the VCM with \textit{unit} coefficients (see $\Cref{tab:relaxCoef}$) on the model-data constraints and \textit{null} coefficients on the parameters produces $V_{||\cdot||_1}(r) \in [0.018,0.024]$. The positive lower bound, determined via semidefinite programming (see {\Cref{eq:rvcm} in \Cref{sec:AppendixA} for details)}, provides proof that the GRI-Mech 3.0 dataset is inconsistent. The optimal relaxations suggest decreasing the lower bound of the 37th QOI by 0.014 (a $0.9\%$ decrease)  and increasing the upper bound of the 36th QOI by 0.010 (a $0.67\%$ increase), as diagrammed in \Cref{fig:GRIavcm1}. Performing a vector consistency analysis with \textit{bound} coefficients results in the interval $[0.012, 0.015]$. This choice of relaxation coefficients reveals that consistency can be reached by decreasing just the lower bound of the 37th QOI by $1.5\%$. 
\begin{figure}[htbp]
  \centering
  \includegraphics[width=\linewidth]{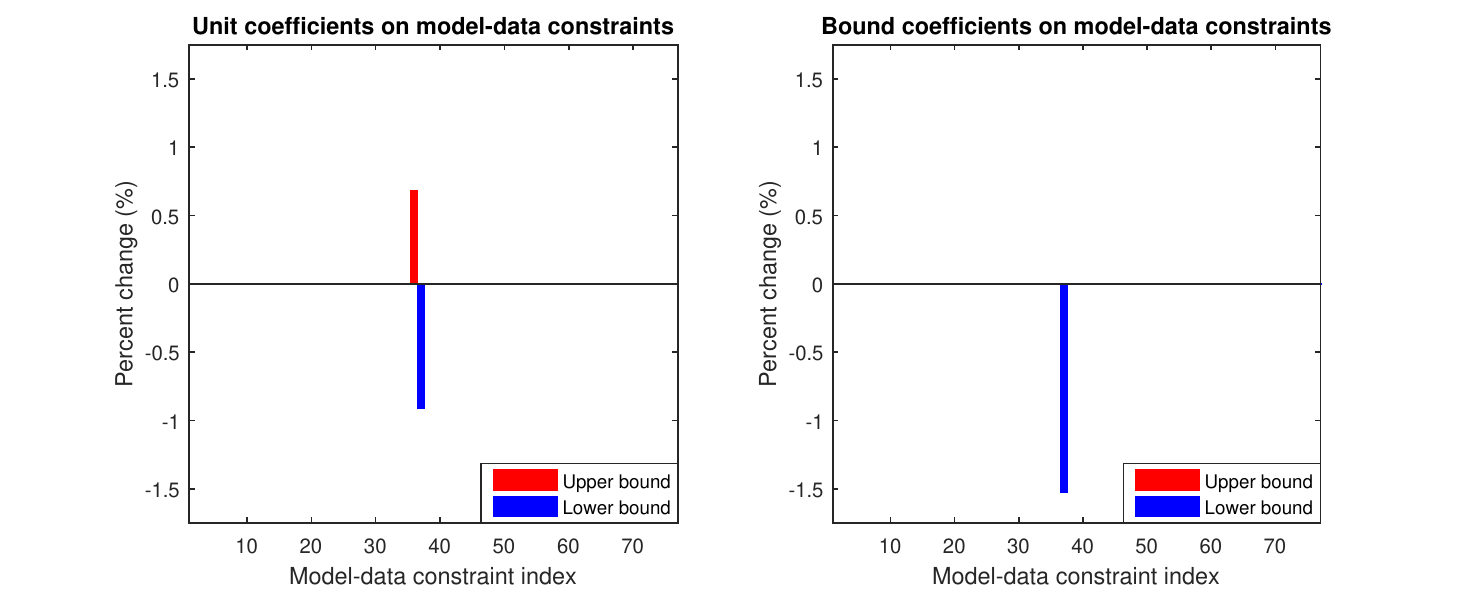}
  \caption{Vector consistency relaxations for the GRI-Mech 3.0 dataset with \textit{unit} coefficients (left) and bound coefficients (right) on the model-data constraints.}
  \label{fig:GRIavcm1}
\end{figure}
\FloatBarrier  
So far, the results match those of the iterative-SCM method. The augmented vector consistency approach, however, allows for new forms of analysis. From the previous paragraph, we have learned that the GRI-Mech 3.0 dataset is nearly consistent; we only need to decrease a single bound by less than $2\%$ to reach consistency. Since the dataset is consistent for the remaining 76 model-data constraints and the inconsistency is attributed to something so minor, it seems likely that the underlying kinetic model is actually valid and the inconsistency is due to QOI 37's specified experimental uncertainty. If QOI 37 is indeed believed accurate, alternative relaxations can be found by testing different relaxation coefficients. For instance, placing a \textit{null} coefficient on the lower bound of constraint 37, \textit{null} coefficients on the parameters, and \textit{unit} coefficients on the remaining model-data constraints produces a VCM in $[0.022,0.037]$. In this case, consistency is reached by increasing the upper bounds of QOIs 35 and 36 by $0.97\%$ and $1.83\%$ respectively. The model-data constraint corresponding to QOI 37 is no longer involved. 

GRI-Mech 3.0's inconsistency can be further explored by alternate coefficient schemes. The initial vector consistency analysis suggested that the upper bound of QOI 36 and the lower bound of QOI 37 were suspect. Suppose all other relaxation coefficients were set to zero. Therefore, we now consider a problem with relaxations to only the two aforementioned bounds. Sampling various coefficient configurations and computing the optimal relaxations produces the trade-off curve, shown in \Cref{fig:tradeoff}.
\begin{figure}[htbp]
  \centering
  \includegraphics[scale = 0.75]{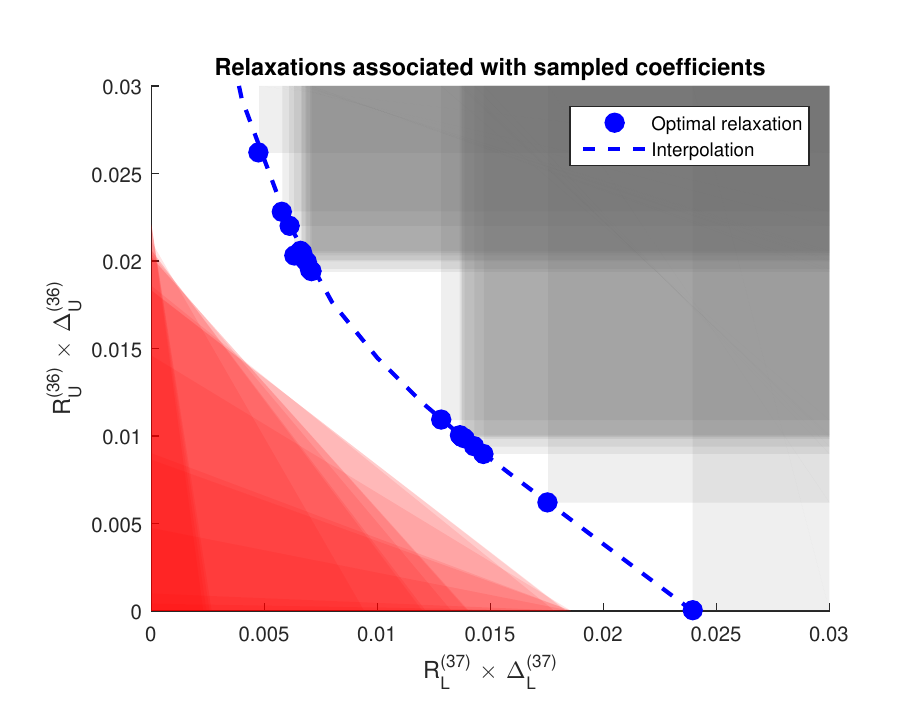}
  \caption{Relaxations for QOIs 37 and 36 due to random selection of relaxation coefficients. Blue dots represent optimal relaxations. The shaded gray region contains other feasible relaxations. Any relaxation in the shaded red region is guaranteed infeasible by a semidefinite program result (see \Cref{sec:AppendixA} for details). }
  \label{fig:tradeoff}
\end{figure}
\FloatBarrier

In the previous examples, we prevented relaxations to the parameter bounds (the prior information) by assigning \textit{null} coefficients. This decision reflected an assumption that those bounds were an accurate and well understood characterization of the parameter uncertainty.
When this is not the case, i.e., when this prior information is not well-known, one may be equally willing, or even prefer, to adjust the parameter bounds. Including \textit{unit} coefficients for both QOI and parameter bounds produces a VCM in $[0, 0.024]$ and identifies the familiar QOIs 36 and 37 as needing relaxation. Attaching \textit{null} coefficients to the QOIs and \textit{unit} coefficients to the parameters, however, produces a VCM in $[0,0.70]$ and requires decreasing only the lower bound of parameter $x_{14}$ by $70\%$. By isolating QOI 37 and parameter $x_{14}$, we can conduct a trade-off analysis similar to the above figure by carefully selecting differing relaxation coefficients. The result of this is shown in $\Cref{fig:tradeoff2}$. It is important to note that in the GRI-Mech 3.0 dataset, the QOI models $\{M_e(x)\}_{e=1}^{77}$ are surrogate models fit over the original parameter range. Therefore, the proper course of action would be to refit these models over the expanded parameter range and reassess consistency. 
\begin{figure}[htbp]
  \centering
  \includegraphics[scale = 0.75]{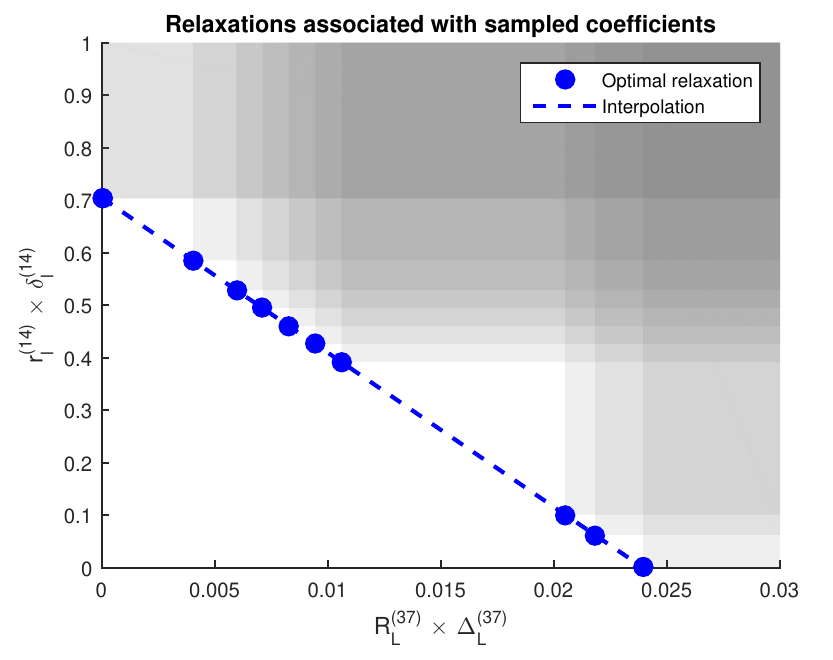}
  \caption{Relaxations for QOI 37 (horizontal axis) and parameter $x_{14}$ (vertical axis) due to  differing relaxation coefficients. Blue dots represent optimal relaxations. The shaded gray region contains other feasible relaxations. There is no region similar to the red one in $\Cref{fig:tradeoff}$ as the semidefinite program produces a lower bound of 0 to the VCM, i.e., the VCM is in $[0, \cdot]$, for all sampled relaxation coefficients.}
  \label{fig:tradeoff2}
\end{figure}
\FloatBarrier

 Computing the VCM with \textit{unit} coefficients for each constraint results in the interval $[0,0.024]$. The semidefinite program approximation, which provides the lower bound of $0$, is non-informative. We cannot conclusively say that the dataset is inconsistent, which contrasts with our earliest analysis where \textit{null} coefficients were applied to the parameters. This example illustrates that using \textit{null} coefficients can be beneficial in proving inconsistency. For instance, the VCM with null coefficients will always produce a value greater than or equal to the VCM without such specification. To see this, consider \Cref{eq:avcm} and note that assigning a \textit{null} coefficient, which effectively removes the corresponding relaxation, is equivalent to fixing that relaxation to zero. As such, using null coefficients can be interpreted as adding ``equal to zero'' constraints to the formulation. Hence, the VCM with null coefficients is a minimization over a further constrained set, giving the result. This statement also carries over to the semidefinite program approximation of \Cref{eq:rvcm}, described in \Cref{sec:AppendixA}.

\subsection{DLR-SynG}
As discussed in \Cref{sec:GRIandDLR}, the DLR-SynG dataset was found to be massively inconsistent and the iterative-SCM method was rather ineffective for consistency analysis. In this application, the augmented VCM fares much better. Assigning \textit{unit} coefficients to the model-data constraints and \textit{null} coefficients to the parameters results in a VCM in $[7.15,12.91]$. This dataset is provably inconsistent. As shown in \Cref{fig:DLRavcm1}, relaxations are suggested to $41$ QOIs. The relatively large number of relaxations and their magnitudes are indicative of a massive inconsistency. Repeating this result with bound coefficients produces a similar outcome, this time with a measure in $[1.13,1.70]$ and relaxing only $37$ QOIs. These results are much improved over the iterative analysis described in \Cref{sec:GRIandDLR}, which, among the strategies attempted, suggested the removal of at least $56$ model-data constraints. 
\begin{figure}[htbp]
  \centering
  \includegraphics[width=\linewidth]{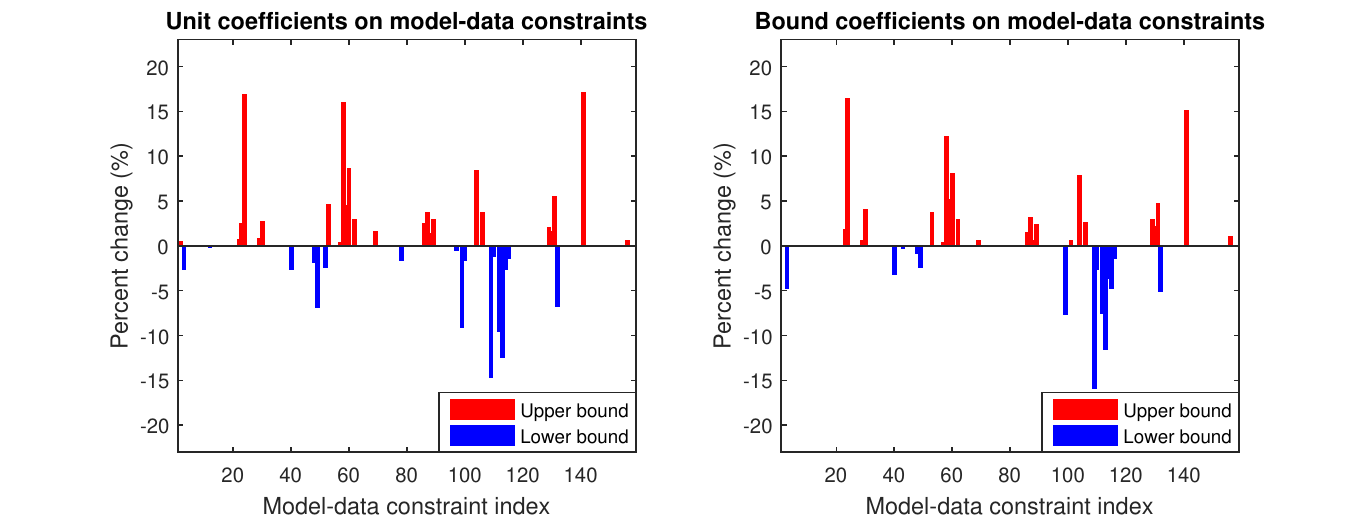}
  \caption{Vector consistency relaxations for the DLR-SynG dataset with \textit{unit} coefficients (left) and bound coefficients (right) on the model-data constraints.}
  \label{fig:DLRavcm1}
\end{figure}
\FloatBarrier
The recommended relaxations are much greater in percentage than in the GRI-Mech 3.0 dataset. For instance, the upper bound of QOI $141$ requires an almost $18\%$ increase. Note that although the earlier iterative-SCM identified QOI $141$ as a potential culprit (\Cref{fig:DLRSensitivity1}), it had assigned many of the QOIs listed in \Cref{fig:DLRavcm1} near-zero sensitivities, suggesting they were not locally influential in resolving inconsistency. As with GRI-Mech 3.0, alternative relaxation coefficient strategies can be used to further the analysis. For example, suppose relaxation was not permitted to QOI $141$. Assigning a \textit{null} coefficient to the upper bound of QOI $141$, \textit{null} coefficients to the parameters, and \textit{unit} coefficients to the remaining model-data constraints produces a VCM in $[8.05, 15.64]$. In this instance, relaxations to $46$ constraints are required to reach consistency (see \Cref{fig:DLRavcm2}).  
\begin{figure}[htbp]
  \centering
  \includegraphics[scale=0.75]{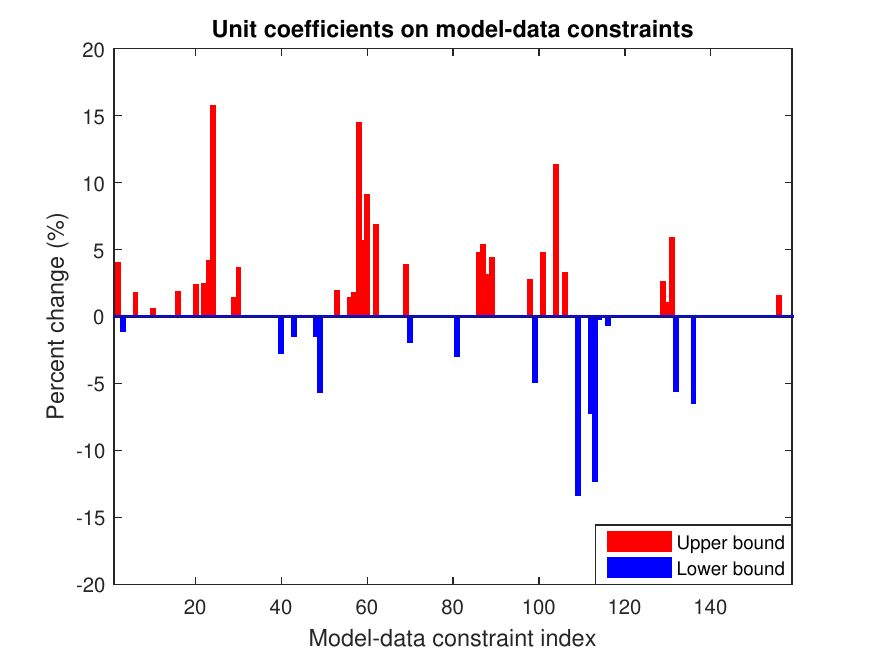}
  \caption{Vector consistency relaxations for the DLR-SynG dataset with a \textit{null} coefficient on the upper bound of QOI \#141, \textit{null} coefficients on the parameters, and \textit{unit} coefficients on the remaining QOIs.}
  \label{fig:DLRavcm2}
\end{figure}
As a final test, implementing \textit{unit} coefficients on both model-data constraints and parameters results in a VCM in $[0.87,9.13]$, relaxing $35$ model-data constraints and $6$ parameter bounds, as shown in \Cref{fig:DLRavcm3}. In contrast, \textit{bound} coefficients for both QOIs and parameters produces a VCM in $[0.37,1.69]$, relaxing $38$ model-data constraints and no parameter bounds.  
\begin{figure}[htbp]
  \centering
  \includegraphics[width=\linewidth]{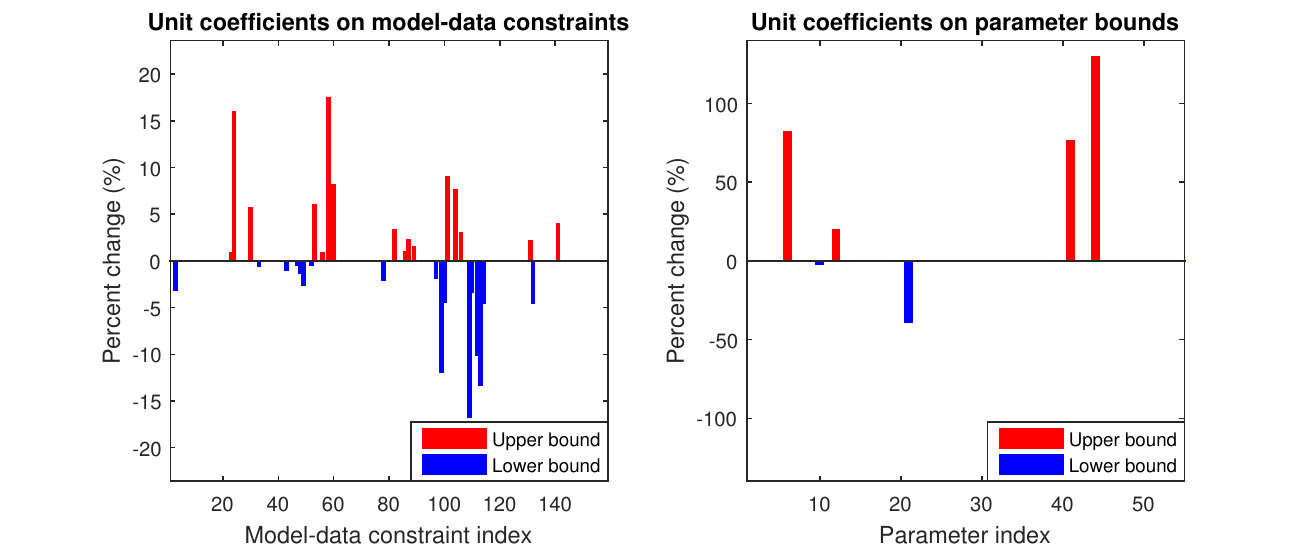}
  \caption{Vector consistency relaxations for the DLR-SynG dataset with \textit{unit} coefficients on both model-data constraints and parameter bounds.}
  \label{fig:DLRavcm3}
\end{figure}
\FloatBarrier
\section{Summary and conclusions}
\label{sec:summary}
Consistency analysis is a fundamental component of the Bound-to-Bound Data Collaboration (B2BDC) methodology and provides a rigorous framework for validating collections of QOI models. In the present work, we develop an alternative to the scalar consistency measure (SCM), which was originally formulated in \cite{feeley04}. Resolving inconsistent datasets using the SCM is indirect and can lead to poor results, particularly when applied to massively inconsistent datasets. To overcome this difficulty, the vector consistency measure (VCM), searches for a sparse set of independent constraint relaxations that leads to consistency. Further, augmenting the VCM with relaxation coefficients allows the user to give preferential treatment to certain models and experimental observations in the dataset. The efficacy of this new strategy is demonstrated on two real-world examples from combustion chemistry: GRI-Mech 3.0 and DLR-SynG. In both cases, the VCM provides new insights and enables a richer form of consistency analysis.

 Consistency measures are tools to accomplish validation. We may suggest the following workflow for practical use of the VCM and SCM. Begin with a model to test against experimental data. Build surrogate models for the specified QOIs and construct a dataset linking the QOI models with experimental bounds. Compute the SCM and its accompanying sensitivities. If consistent, a parameter configuration exists for which the model predictions match the experimental data. If inconsistent and the sensitivities cannot immediately resolve the issue, compute the VCM. Relaxation coefficients can be used to devise a strategy to explore and reconcile the inconsistency. The presence of a massive inconsistency could be indicative of model deficiencies. Once the models and/or observations are revised by domain scientists, the process can be repeated using the updated dataset. We emphasize that the SCM and VCM are quantitative measures of ``model fits data''--- numerical tools to assist domain scientists in resolving inconsistency. In this spirit, we reemphasize that both the bounds and models in B2BDC are tentatively entertained.

\section*{Acknowledgments}
This material is based upon work supported by the U.S. Department of Energy, National Nuclear Security Administration, under Award Number DE-NA0002375. The views and opinions of authors expressed herein do not necessarily state or reflect those of the United States Government or any agency thereof.

\newpage
\bibliographystyle{siamplain}
\bibliography{references}

\newpage
\appendix
\section{Semidefinite programming approximations to the scalar and vector consistency measures} 
\label{sec:AppendixA}
For quadratic models, \Cref{eq:scm} and \Cref{eq:avcm} can be represented more generally as,
\begin{equation}
\label{eq:scm-appendix}
\begin{aligned}
&\textbf{SCM:} & \underset{x \in \mathbb{R}^n,\gamma}{\max} & \qquad \gamma \\
& &\text{s.t.} & \qquad \begin{bmatrix}1 \\ x \end{bmatrix}^\intercal 
Q_e
 \begin{bmatrix}1 \\ x \end{bmatrix} \leq - w_e \gamma \qquad e = 1,...,N \\
& & & \qquad a_i^\intercal \begin{bmatrix}1 \\ x \end{bmatrix} \leq 0\qquad i = 1,...,m \\   
\end{aligned}
\end{equation}
where $Q_e$ is the coefficient matrix of a model (bounds are subsumed into the constant term), the parameter prior $\mathcal{H}$ is defined by constraints involving vectors $a_i$, and $w_e \geq 0$. 
\begin{equation}
\label{eq:vcm-appendix}
\begin{aligned}
&\textbf{VCM:} & \underset{x \in \mathbb{R}^n,\Delta \in \mathbb{R}^N, \delta \in \mathbb{R}^{m}}{\min} & \qquad \textbf{1}_N^\intercal \Delta + \textbf{1}_m^\intercal \delta \\
& &\text{s.t.} & \qquad \begin{bmatrix}1 \\ x \end{bmatrix}^\intercal 
Q_e
 \begin{bmatrix}1 \\ x \end{bmatrix} \leq R_e \Delta_e \qquad e = 1,...,N \\
& & & \qquad  a_i^\intercal \begin{bmatrix}1 \\ x \end{bmatrix} \leq  r_{i} \delta_{i} \qquad i = 1,...,m \\ 
& & & \qquad \Delta \geq 0 \\
& & & \qquad \delta \geq 0  
\end{aligned}
\end{equation}
where $\textbf{1}_d$ is a $d-$dimensional vector of ones. In Feeley et al. \cite[Appendix]{feeley04}, Lagrangian duality was used to formulate a semidefinite program (SDP) that bounds the \textbf{SCM} from above, producing the $\overline{C}_D$ defined in \Cref{sec:sense}. Under mild technical conditions (constraint qualifications), the rank relaxation construction of an SDP presented in this section provides an equivalent solution to a Lagrangian approach \cite{vandenberghe96,fujie97}.
 
For the \textbf{SCM}, note that the constraints for each $e$ and $i$ can be rewritten as
\begin{equation}
\label{eq:scmconstr}
\begin{aligned}
& trace \left( 
\begin{bmatrix} Q_e & 0 \\ 0 & 0 \end{bmatrix}
\begin{bmatrix}1 \\ x \\ \gamma \end{bmatrix} \begin{bmatrix}1 \\ x \\ \gamma \end{bmatrix}^\intercal \right) +trace \left( \begin{bmatrix} 0 & 0 & 0.5 w_e \\ 0 & 0 & 0 \\ 0.5 w_e & 0 & 0 \end{bmatrix}
\begin{bmatrix}1 \\ x \\ \gamma \end{bmatrix} \begin{bmatrix}1 \\ x \\ \gamma \end{bmatrix}^\intercal \right)  
\leq 0 \\
& trace \left( \begin{bmatrix}a_i(1) & 0.5 a_i^\intercal (2{:}n{+}1) & 0 \\ 0.5 a_i(2{:}n{+}1) & 0 & 0\\ 0 & 0 & 0 \end{bmatrix} \begin{bmatrix}1 \\ x \\ \gamma \end{bmatrix}\begin{bmatrix}1 \\ x \\ \gamma \end{bmatrix}^\intercal \right) \leq 0 \\
\end{aligned}
\end{equation}
where $a(j{:}k)$ denotes the $j$th through $k$th entries of vector $a$ and similarly for entries and submatrices of a given matrix. Let $\mathcal{S}^d \subset \mathbb{R}^{d \times d}$ refer to the subset of symmetric matrices.  Rank relaxation (or SDP relaxation) \cite{fujie97} is carried out by noting for a vector $v \in \mathbb{R}^d$, 
\begin{equation} 
Z = \begin{bmatrix}1 \\ v \end{bmatrix} \begin{bmatrix}1 \\ v \end{bmatrix}^\intercal \quad \Leftrightarrow \quad Z(1,1) = 1, Z \succeq 0, rank(Z) = 1
\end{equation}
Dropping the $rank$ constraint, inserting the matrix $Z$ appropriately in \Cref{eq:scmconstr}, and then replacing the constraints in \Cref{eq:scm-appendix} yields an SDP that bounds the \textbf{SCM} from above. This bound can be tightened by including additional constraints constructed from the problem data (i.e., from the given constraints) \cite{sheralibook, anstreicher09, ashraphijou16}. This procedure is more generally characterized by Positivstellensatz refutations \cite{parrilo03}. For a given collection of linear constraints, 
\begin{equation} 
\label{eq:redconstr}
C ^\intercal \begin{bmatrix} 1 \\ v \end{bmatrix} \leq 0 \quad \Rightarrow \quad C^\intercal \begin{bmatrix} 1 \\ v \end{bmatrix} \begin{bmatrix} 1 \\ v \end{bmatrix}^\intercal C \geq 0
\end{equation}
where the inequalities on both sides are applied entrywise. 

Let $A = \begin{bmatrix} a_1 & a_2 & ... & a_m \end{bmatrix} \in \mathbb{R}^{(1+n) \times m}$. An SDP approximation of the \textbf{SCM} is built by letting $v = \begin{bmatrix} x^\intercal & \gamma \end{bmatrix}^\intercal$ and combining the above equations to produce \Cref{eq:rscm}. 
\begin{equation}
\label{eq:rscm}
\begin{aligned}
&\textbf{rSCM:} & \underset{Z \in \mathcal{S}^{2+n}}{\max} & \qquad Z(2{+}n,1) \\
& &\text{s.t.} & \qquad trace \left( Q_e Z(1{:}1{+}n,1{:}1{+}n) \right) + w_e Z(2+n,1) \leq 0 \qquad e = 1,...,N \\
& & & \qquad A^\intercal Z(1{:}1{+}n,1) \leq 0 \\
& & & \qquad A^\intercal Z(1{:}1{+}n,1{:}1{+}n) A \geq 0 \\
& & & \qquad Z(1,1) = 1 \\
& & & \qquad Z \succeq 0    
\end{aligned}
\end{equation}
In cases where $m$ is large, including all of the extra constraints listed in \Cref{eq:redconstr} (accounting for symmetry, repetitions, and implications of $Z \succeq 0$) can become computationally unmanageable, but for the theoretical purpose of this section all are considered. In Ashraphijou et al. \cite{ashraphijou16}, only constraints that are binding at optimality are added. In the examples of \Cref{sec:caseStudy}, only a subset are included in the computation.
 
An SDP relaxation of the \textbf{VCM} proceeds similarly, with $v = \begin{bmatrix} x^\intercal & \Delta^\intercal & \delta^\intercal \end{bmatrix}^\intercal$. For notational convenience, let $d=1{+}n{+}N{+}m$ and $diag(r)$ denote the diagonal matrix with the vector $r$ along the diagonal.    
\begin{equation}
\label{eq:rvcm}
\begin{aligned}
&\textbf{rVCM:} & \underset{Y \in \mathcal{S}^{d}}{\min} & \qquad \textbf{1}_{N+m}^\intercal Y(2{+}n{:}d,1) \\
& &\text{s.t.} & \qquad trace \left( Q_e Y(1{:}1{+}n,1{:}1{+}n) \right) - R_e Y(1{+}n{+}e,1) \leq 0 \qquad e = 1,...,N \\
& & & \qquad A^\intercal Y(1{:}1{+}n,1) - diag(r)Y(2{+}n{+}N{:}d,1) \leq 0 \\
& & & \qquad A^\intercal Y(1{:}1{+}n,1{:}1{+}n) A - diag(r) Y(2{+}n{+}N{:}d,1{:}1{+}n) A-\\ 
& & & \qquad \qquad A^\intercal Y(1{:}1{+}n,2{+}n{+}N{:}d)diag(r) +\\
& & & \qquad \qquad diag(r) Y(2{+}n{+}N{:}d, 2{+}n{+}N{:}d) diag(r) \geq 0\\ 
& & & \qquad Y(2{+}n{:}d,1) \geq 0 \\ 
& & & \qquad Y(1,1) = 1 \\
& & & \qquad Y \succeq 0
\end{aligned}
\end{equation}
As mentioned, $\textbf{SCM} \leq \textbf{rSCM}$ and similarly $\textbf{rVCM} \leq \textbf{VCM}$. An immediate consequence is that, for given relaxation coefficients $R$ and $r$, if $\|\Delta\|_1+\|\delta\|_1 < \textbf{rVCM}$, then $\Delta$ and $\delta$ cannot be feasible with respect to $\textbf{VCM}$ as it contradicts its minimality. Hence, expanding bounds by the amounts $R_e \Delta_e$ and $r_i \delta_i$ does not lead to consistency (e.g., the infeasible region in \Cref{fig:tradeoff}). 

Using \textit{null} coefficients can be helpful in proving inconsistency. Consider two problems, $\textbf{rVCM}_0$ and $\textbf{rVCM}_1$, where $\textbf{rVCM}_0$ shares all but a subset of \textit{null} coefficients with $\textbf{rVCM}_1$. Let $\mathcal{J}_R$ and $\mathcal{J}_r$ contain the indices of the respective \textit{null} coefficients of  $\textbf{rVCM}_0$, i.e., $e \in \mathcal{J}_R \Leftrightarrow R_e = 0$ and $i \in \mathcal{J}_r \Leftrightarrow r_i = 0$. Moreover, let $\mathcal{J} = \{1{+}n{+}\mathcal{J}_R\} \cup \{1{+}n{+}N{+}\mathcal{J}_r\}$ be the adjusted set of indices. Suppose $Y_0^\star$ is the minimizer of  $\textbf{rVCM}_0$ and define $Y \in \mathcal{S}^d$ by
\begin{equation}
Y(i,j) = 
\begin{cases} 
0 & i \in \mathcal{J} \text{ or } j \in \mathcal{J} \\
Y_0^\star(i,j) & \text{ otherwise} \\
\end{cases}
\end{equation} 
As defined, $Y$ is feasible with respect to  $\textbf{rVCM}_0$ and $\textbf{1}_{N+m}^\intercal Y(2{+}n{:}d,1) \leq \textbf{1}_{N+m}^\intercal Y_0^\star(2{+}n{:}d,1)$. Since these summations are of only nonnegative numbers, we must have that $Y_0^\star(i,1) = 0$ for $i \in \mathcal{J}$ and thus $\textbf{1}_{N+m}^\intercal Y(2{+}n{:}d,1) = \textbf{1}_{N+m}^\intercal Y_0^\star(2{+}n{:}d,1)$. Furthermore, note that $Y$ is also feasible with respect to  $\textbf{rVCM}_1$, implying $\textbf{rVCM}_1 \leq \textbf{rVCM}_0$. Since $\textbf{rVCM}>0$ proves inconsistency, we have demonstrated that including extra \textit{null} coefficients may aid in this task.

Finally, the \textbf{rSCM} is a stronger tool than \textbf{rVCM} for proving inconsistency. 
\begin{theorem}
\label{thm:thm1}
If $\textbf{rSCM} \geq 0$, then $\textbf{rVCM} = 0$ for any choice of relaxation coefficients.
\end{theorem}     
\begin{proof}
Suppose $\textbf{rSCM} \geq 0$. Let $Z^\star \in \mathcal{S}^{2+n}$ be the maximizer. Let $R$ and $r$ be some choice of relaxation coefficients. Define $Y \in \mathcal{S}^d$ by
\begin{equation}
Y(i,j) = \begin{cases}
Z^\star(i,j) & i,j \leq 1{+}n \\
0 & \text{ otherwise} \end{cases}
\end{equation}
Since $Z^\star$ is feasible with respect to \Cref{eq:rscm} and  $Z^\star(2+n,1) \geq 0$, $Y$ must be feasible with respect to \Cref{eq:rvcm}. Moreover, $\textbf{1}_{N+m}^\intercal Y(2{+}n{:}d,1) = 0$. Since \textbf{rVCM} is always nonnegative, $\textbf{rVCM} = 0$. 
\end{proof}

The contrapositive of \Cref{thm:thm1} states that if $\textbf{rVCM} > 0$ for some choice of relaxation coefficients, then $\textbf{rSCM}<0$. Hence, the class of datasets for which \textbf{rVCM} proves inconsistency is contained in the class of datasets for which \textbf{rSCM} proves inconsistency. In \Cref{sec:AppendixB}, we present a numerical example where $\textbf{rVCM} = 0$ and $\textbf{rSCM}<0$, showing that the converse of \Cref{thm:thm1} does not hold. This justifies the use of the \textbf{SCM} in the suggested workflow of \Cref{sec:summary}. 

We conclude this section by noting that the above semidefinite programming techniques can be extended beyond quadratic models to polynomial models using techniques described in \cite{lasserre01,parrilo03}. A general discussion of these methodologies applied to the B2BDC framework can be found in \cite{seiler06}.    

\section{A numerical example showing the converse of \Cref{thm:thm1} does not hold} 
\label{sec:AppendixB}
Consider an artificial dataset (of the form described in \Cref{eq:scm-appendix} and \Cref{eq:vcm-appendix}):
\begin{equation}
\begin{aligned}
& Q_1 = \begin{bmatrix} 0.0881 & 0.460 & 0.4769 \\ 0.460 & 0.5613 & 0.4948 \\ 0.4769 & 0.4948 & 0.3550 \end{bmatrix}  
& & Q_2 = \begin{bmatrix} 0.2448 & 0.1876 & 0.1492 \\ 0.1876 & 0.2664 & 0.7218 \\ 0.1492 & 0.7218 & 0.1476 \end{bmatrix} \\
& a_1 = \begin{bmatrix} 0.9207 \\ 0.9295 \\ 0.1368 \end{bmatrix} 
& & a_2 = \begin{bmatrix} 0.8716 \\ 0.0124 \\ 0.7220 \end{bmatrix} \\
& R_1 = R_2 = r_1 = r_2 =1
\end{aligned}
\end{equation}
where $n=2$, $N=2$, and $m=2$. In this example, the only extra constraint added to the \textbf{rSCM} is $a_1^\intercal Z(1{:}3,1{:}3) a_2 \geq 0$ as the other constraints are implied by $Z \succeq 0$, and similarly for the \textbf{rVCM}. Using CVX \cite{cvx,gb08}, we find that $\textbf{rSCM} = -1.0857$ and $\textbf{rVCM} = 0$. To protect against issues of precision in the numerical evaluation of $\textbf{rVCM}$, we round the minimizer $Y^\star$ to the second digit and find that it is still feasible with some margin. 

\section{Linear example continued for the augmented VCM} 
\label{sec:AppendixC}
With relaxation coefficients, the constraint becomes
\begin{equation}
\label{eqC:counterEx}
 \begin{bmatrix} 1.5 \\ -1 \end{bmatrix} x - \begin{bmatrix} 1 \\ 1\end{bmatrix} + \alpha \begin{bmatrix} 1 \\ 0 \end{bmatrix}
\leq \begin{bmatrix} r_1 \delta_1 \\ r_2 \delta_2 \end{bmatrix}.
\end{equation}
Then, the optimal relaxation $\delta^\star$ is characterized by the intersection of two regions,
\begin{equation}
\begin{gathered}
 K_\alpha = \{ y \in \mathbb{R}^2: \ \exists \delta \text{ such that } \|\delta\|_1 \leq c^\star, \ \ \begin{bmatrix} y_1 \\ y_2 \end{bmatrix} \leq \begin{bmatrix} r_1 \delta_1 \\ r_2 \delta_2 \end{bmatrix} \} \\
 G_\alpha = \{y \in \mathbb{R}^2 : y = Ax-b+ \alpha t, x \in \mathbb{R} \},
\end{gathered}
\end{equation}
where $c^\star$ is the smallest value such that the intersection is non-empty (i.e., the solution of \Cref{eqC:counterEx}). Suppose $r_1 = 2$ and $r_2 = 1$. With $\alpha = 4$, $\delta^\star = \begin{bmatrix} 0.75 \\ 0 \end{bmatrix}$, indicating we have now correctly associated the error with the $y_1$ axis. This result is shown in the figure below.  
\begin{figure}[htbp]
  \centering
  \includegraphics[scale = 0.75]{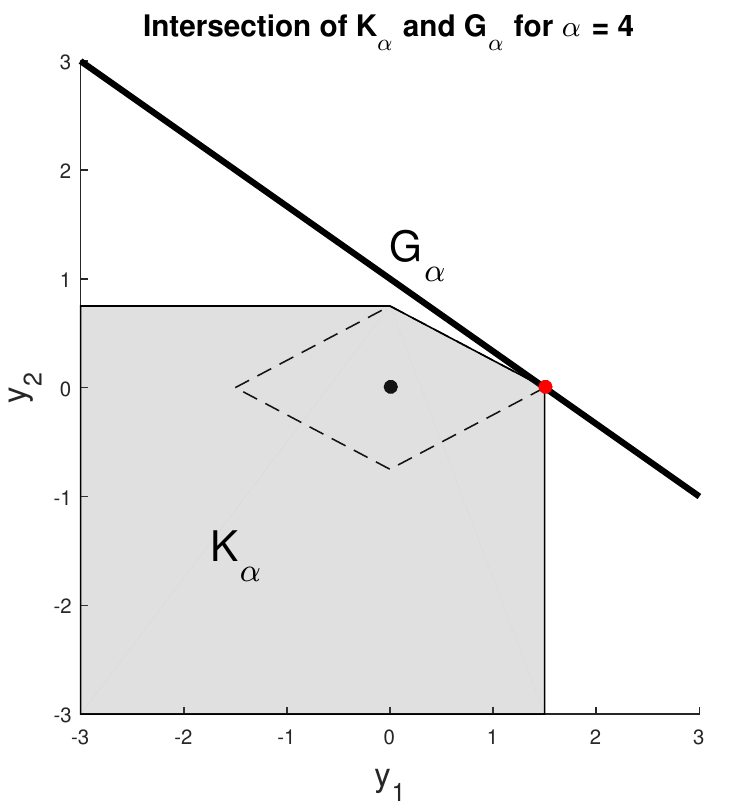}
  \caption{Augmented VCM applied to the linear example with $r_1 = 2$ and $r_2 = 1$. The point of intersection between the two regions, shown in red, is $(y_1, y_2) = (r_1 \delta_1^\star, r_2 \delta_2^\star)$. }
  \label{fig:counterEx2}
\end{figure}
\FloatBarrier

\end{document}